\documentclass[12pt,reqno]{article}

\usepackage[usenames]{color}
\usepackage{amssymb}
\usepackage{amsmath}
\usepackage{amsthm}
\usepackage{amsfonts}
\usepackage{amscd}
\usepackage{graphicx}

\usepackage{booktabs}

\usepackage[colorlinks=true,
linkcolor=webgreen,
filecolor=webbrown,
citecolor=webgreen]{hyperref}

\definecolor{webgreen}{rgb}{0,.5,0}
\definecolor{webbrown}{rgb}{.6,0,0}

\usepackage{color}
\usepackage{fullpage}
\usepackage{float}

\usepackage{graphics}
\usepackage{latexsym}

\DeclareMathOperator{\Li}{Li}

\usepackage{xcolor}

\newcommand{\braces}{\genfrac{\lbrace}{\rbrace}{0pt}{}}

\allowdisplaybreaks

\begin{document}

\theoremstyle{plain}
\newtheorem{theorem}{Theorem}
\newtheorem{proposition}[theorem]{Proposition}
\newtheorem{corollary}[theorem]{Corollary}
\newtheorem{lemma}[theorem]{Lemma}
\theoremstyle{definition}
\newtheorem{example}[theorem]{Example}
\newtheorem{remark}{Remark}

\numberwithin{equation}{section}
\numberwithin{theorem}{section}
\numberwithin{remark}{section}

\begin{center}
\vskip 1cm
{\LARGE\bf Double sums involving binomial coefficients and special numbers}

\vskip 1cm

{\large
Kunle Adegoke \\
Department of Physics and Engineering Physics \\ Obafemi Awolowo University \\ 220005 Ile-Ife, Nigeria \\
\href{mailto:adegoke00@gmail.com}{\tt adegoke00@gmail.com} \\
ORCID: 0000-0002-3933-0459

\vskip .25 in

Robert Frontczak \\
Independent Researcher \\ 72764 Reutlingen,  Germany \\
\href{mailto:robert.frontczak@web.de}{\tt robert.frontczak@web.de}\\
ORCID: 0000-0002-9373-5297

\vskip .25 in

Karol Gryszka \\
Institute of Mathematics \\ University of the National Education Commission, Krakow \\
Podchor\c{a}\.{z}ych 2, 30-084 Krak{\'o}w, Poland\\
\href{mailto:email}{\tt karol.gryszka@uken.krakow.pl} \\
ORCID: 0000-0002-3258-3330

\vskip .25 in
}

\end{center}

\vskip .2 in

\begin{abstract}
In this paper, we find an elementary approach for double sums where the inner sum is binomial but incomplete. 
We apply our core identity and its relatives to double sums involving famous numbers such as harmonic numbers, 
Fibonacci numbers, Stirling numbers and $r$-Stirling numbers of the second kind.
\end{abstract}

\noindent 2020 {\it Mathematics Subject Classification}:
Primary 05A10; Secondary 11B39, 11B65, 11B73.

\noindent \emph{Keywords:}
Double sum, binomial coefficient, harmonic number, Fibonacci number, Stirling number.

\section{Motivation}

The binomial coefficients are defined, for non-negative integers $n$ and $k$, by
\begin{equation*}
\binom {n}{k} =
\begin{cases}
\frac{n!}{k!(n - k)!}, & \text{$n\geq k$};\\
0, & \text{$n<k$}.
\end{cases}
\end{equation*}
More generally, for complex numbers $r$ and $s$, they are defined by
\begin{equation*}
\binom {r}{s} = \frac{\Gamma (r+1)}{\Gamma (s+1) \Gamma (r-s+1)},
\end{equation*}
where the Gamma function, $\Gamma(z)$, is defined for $\Re(z)>0$ by the integral \cite{Srivastava}
\begin{equation*}
\Gamma (z) = \int_0^\infty e^{- t} t^{z - 1}\,dt.
\end{equation*}

Sum identities involving binomial coefficients are ubiquitous in probability, combinatorics and number theory.
These sums are so prominent that whole textbooks have been devoted to describe the techniques for their evaluation 
\cite{Gould,Graham,Riordan,Wilf}, the most prominent being generating function, contour integration, the usage of analytical and combinatorial arguments, and the application of the Wilf-Zeilberger algorithm. Single sums over a reciprocal of a binomial coefficient have also been studied extensively. Recent articles dealing with these forms include \cite{Adegoke1,Adegoke2,Batir,Belbachir,Borwein,Mansour,Sofo,Witula}, 
and several others. One common technique to tackle these forms is to express the binomial coefficient $\binom{n}{k}^{-1}$ as a Beta integral. 
Another approach is to apply differentiation to generalized binomial coefficients. \\

Double binomial sums have also been studied in the recent past. A general telescoping approach for double summation was outlined 
by Chen et al. in \cite{Chen}. Chu \cite{Chu1,Chu2} studied two prominent double sums in 2017 and 2018, respectively. 
Our paper was particularly inspired by the observation that the incomplete binomial sum 
\begin{equation}\label{incomplete}
\sum_{j=0}^{k} \binom{n}{j} x^j, \quad x\in\mathbb{C},\, k < n,
\end{equation}
has no simple closed form. If we work with the Gaussian hypergeometric function ${}_2F_1$ defined by
\begin{equation*}
{}_2F_1(a,b;c;z) = \sum_{n=0}^\infty \frac{(a)_n (b)_n}{(c)_n} \frac{z^n}{n!},
\end{equation*}
where $(a)_n$ is the Pochhammer symbol given by $(a)_n = \Gamma(a+n)/\Gamma(a)$, and $a,b,c$ are complex numbers, 
then there is the relation \cite{Stenlund}
\begin{equation*}
\sum_{j=0}^{k} \binom{n}{j} x^j = \frac{x^{k+1}}{x+1} \binom{n}{k} \,{}_2F_1\left(1,n+1;n+1-k;\tfrac{1}{x+1}\right).
\end{equation*}
This relation is useful under some circumstances but it does not allow for an elementary treatment of sums involving the incomplete binomial sum. In this paper, we find an elementary approach for double sums where the inner sum is binomial but incomplete as in \eqref{incomplete}. 
We apply our core identity to double sums involving a range of famous numbers such as harmonic numbers, Fibonacci numbers, Stirling numbers and $r$-Stirling numbers of the second kind.

\section{Preliminaries}

This section contains the definitions and basic relations for the quantities used in the main text. 

The Fibonacci numbers $F_n$ and the Lucas numbers $L_n$ are defined, for \text{$n\in\mathbb Z$}, through the recurrence relations
\begin{align*}
F_n = F_{n-1}+F_{n-2},& \quad n\geq 2,\quad F_0=0,\, F_1=1,\\
L_n = L_{n-1}+L_{n-2},& \quad  n\geq 2, \quad L_0=2, \, L_1=1.
\end{align*}
The Binet formulas for these sequences are
\begin{equation*}
F_n = \frac{\alpha^n - \beta^n}{\alpha - \beta}, \qquad L_n = \alpha^n + \beta^n, \quad n\in\mathbb Z,
\end{equation*}
with $\alpha=\frac{1+\sqrt 5}2$ being the golden ratio and $\beta=-\frac{1}{\alpha}$. For negative subscripts we have 
$F_{-n} = (-1)^{n-1}F_n$ and $L_{-n} = (-1)^n L_n$. These famous sequences are indexed as sequences {A000045} and {A000032} 
in the On-Line Encyclopedia of Integer Sequences \cite{OEIS}. A huge amount of further information about them can be found 
in the books by Koshy \cite{Koshy} and Vajda \cite{Vajda}, for instance. Also recall that the Gibonacci sequences has the same recurrence relation as the Fibonacci sequence but starts with arbitrary initial values, i.e.,
\begin{equation*}
G_k = G_{k - 1} + G_{k - 2},\qquad (k \ge 2),
\end{equation*}
with $G_0$ and $G_1$ arbitrary numbers (usually integers) not both zero. When $G_0=0$ and $G_1=1$ then $G_n=F_n$, and when 
$G_0=2$ and $G_1=1$ then $G_n=L_n$, respectively. The sequence obeys the generalized Binet formula
\begin{equation*}
G_n = A\alpha^n + B\beta^n,
\end{equation*}
where $A=\tfrac{G_1-G_0\beta}{\alpha-\beta}$ and $B=\tfrac{G_0\alpha-G_1}{\alpha-\beta}$. \\

Harmonic numbers $H_s$ and odd harmonic numbers $O_s$ are defined for $0\ne s\in\mathbb C\setminus\mathbb Z^{-}$ 
by the recurrence relations
\begin{equation*}
H_s = H_{s - 1}  + \frac{1}{s} \qquad \text{and} \qquad O_s = O_{s - 1} + \frac{1}{2s - 1},
\end{equation*}
with $H_0=0$ and $O_0=0$. One way to generalize this definition is to consider harmonic numbers $H_s^{(m)}$ 
and odd harmonic numbers $O_s^{(m)}$ of order $m\in\mathbb C$ that are defined by
\begin{equation*}
H_s^{(m)} = H_{s - 1}^{(m)} + \frac{1}{s^m} \qquad \text{and} \qquad O_s^{(m)} = O_{s - 1}^{(m)} + \frac{1}{(2s - 1)^m},
\end{equation*}
with $H_0^{(m)}=0$ and $O_0^{(m)}=0$ so that $H_s=H_s^{(1)}$ and $O_s=O_s^{(1)}$. The recurrence relations imply that 
if $s=n$ is a non-negative integer, then
\begin{equation*}
H_n^{(m)} = \sum_{j = 1}^n \frac{1}{j^m} \qquad \text{and} \qquad O_n^{(m)} = \sum_{j = 1}^n \frac{1}{(2j - 1)^m}.
\end{equation*}
Harmonic numbers are connected to the digamma function $\psi(z)=\Gamma'(z)/\Gamma(z)$ through the fundamental relation
\begin{equation*}
H_z = \psi(z + 1) + \gamma,
\end{equation*}
where $\gamma$ is the Euler-Mascheroni constant. 

We conclude this section with a definition of Stirling numbers of the second kind. For integers $n$ and $k$ the Stirling numbers of the second kind, denoted by $\braces nk$, are the coefficients in the expansion
$$
z^n = \sum_{k=0}^n \braces{n}{k} (z)_k,
$$
where $(z)_k$ is the falling factorial defined by $(z)_{0}=1,(z)_{n}=z (z-1)\cdots (z-n+1)$ for $n>0$. These numbers count the number of ways to partition a set of $n$ elements into exactly $k$ nonempty subsets. We will often make use of the property
$$
\braces{n}{k} = 0, \quad \text{if $k>n$.}
$$
The exponential generating function of $\braces{n}{k}$ equals
$$
\sum_{n\geq k} \braces{n}{k} \frac{t^{n}}{n!}=\frac{1}{k!} \left( e^{t}-1\right)^{k}.
$$
Let $r$ be a positive integer. The $r$-Stirling numbers of the second kind $\braces{n+r}{k+r}_r$
are defined by the exponential generating function 
$$
\sum_{n\geq k} \braces{n+r}{k+r}_r \frac{t^{n}}{n!}=\frac{1}{k!} e^{rt} \left( e^{t}-1\right)^{k}.
$$
We also have
$$
(z+r)^n = \sum_{k=0}^n \braces{n+r}{k+r} (z)_k, \qquad n\geq 0.
$$
Details about these special numbers can be found in Laissaoui and Rahmani~\cite{laissaoui17} and references therein.

\section{The main result}

\begin{theorem}\label{main_thm1}
For all complex numbers $x$ and $y$ we have
\begin{equation}\label{main_id1}
\sum_{k=0}^n x^k \left (\sum_{j=0}^k \binom{n}{j} y^j\right ) = \frac{1}{1-x}\left ( (1+xy)^n - x^{n+1} (1+y)^n \right ).
\end{equation}
\end{theorem}
\begin{proof}
    We have the following chain of equalities:
    \begin{align*}
        (1-x)\sum_{k=0}^n x^k \sum_{j=0}^k\binom{n}{j}y^j&=1-x^{n+1}\sum_{j=0}^n\binom{n}{j}y^j+\sum_{k=1}^{n-1}(x^k-x^{k+1})\sum_{j=0}^k\binom{n}{j}y^j\\
        &=1+\sum_{k=1}^{n}x^k\left(\sum_{j=0}^{k}\binom{n}{j}y^j-\sum_{j=0}^{k-1}\binom{n}{j}y^j\right)-x^{n+1}(1+y)^n\\
        &=1+\sum_{k=1}^nx^k\binom{n}{k}y^k-x^{n+1}(1+y)^n\\
        &=(1+xy)^n-x^{n+1}(1+y)^n.
    \end{align*}
    When $x=1$, both ends equal $0$. Assuming $x\neq 1$ and dividing by $1-x$ we complete the proof.
\end{proof}

We proceed with two important observations dealing with variations of our main result. 
The first variation is based on an equivalent form of the incomplete binomial sum stated as Identity (1.9) by Gould in his book \cite{Gould} 
$$\sum_{j=0}^k \binom{n}{j} y^j = \sum_{j=0}^k \binom{k-n}{j} (1+y)^{k-j}(-y)^j.$$
It now follows that Theorem \ref{main_thm1} can be restated as follows.
\begin{theorem}
For all complex numbers $x$ and $y$ we have
\begin{equation*}
\sum_{k=0}^n x^k \left(\sum_{j=0}^k\binom{k-n}{j}(1+y)^{k-j}(-y)^j \right ) = \frac{1}{1-x}\left ( (1+xy)^n - x^{n+1} (1+y)^n \right ).
\end{equation*}
\end{theorem}
The second variation is concerned with the trigonometric versions of Theorem \ref{main_thm1}. Replacing $x$ with $e^{ix},i=\sqrt{-1},$
using that
$$\frac{1}{1-e^{ix}} = \frac{1}{2}\left ( 1 + i \cot\left(\frac{x}{2}\right )\right ),$$
and comparing the real and imaginary parts we get the trigonometric variants.
\begin{theorem}
For all complex numbers $x$ and $y$ such that $x\notin 2\pi \mathbb{Z}$ we have
\begin{align*}
\sum_{k=0}^n \cos(kx) \sum_{j=0}^k \binom{n}{j} y^j &= \frac{1}{2} \Big (\sum_{k=0}^n \binom{n}{k} y^k \left ( \cos(kx) - \sin(kx)\cot(x/2) \right ) \nonumber \\
&\qquad - (1+y)^n \left ( \cos((n+1)x) - \sin((n+1)x)\right ) \Big ),
\end{align*}
and
\begin{align*}
\sum_{k=0}^n \sin(kx) \sum_{j=0}^k \binom{n}{j} y^j &= \frac{1}{2} \Big (\sum_{k=0}^n \binom{n}{k} y^k \left ( \sin(kx) + \cos(kx)\cot(x/2) \right ) \nonumber \\
&\qquad - (1+y)^n \left ( \sin((n+1)x) + \cos((n+1)x)\right ) \Big ).
\end{align*}
\end{theorem}

Now we state some results that can be drawn from the main result \eqref{main_id1}. 

\begin{corollary}\label{cor:3.4}
With $x$ and $y$ being complex numbers we have the following identities:
\begin{align}
\sum_{k=0}^n x^k \left (\sum_{j=0}^k \binom{n}{j} \right ) &= \frac{1}{1-x}\left ( (1+x)^n - 2^nx^{n+1} \right ), \label{thm1_cor_id1} \\
\sum_{k=0}^n x^k \left (\sum_{j=0}^k \binom{n}{j} (-1)^j \right ) &= (1-x)^{n-1}, \label{thm1_cor_id2} \\
\sum_{k=0}^n \left (\sum_{j=0}^k \binom{n}{j} y^j \right ) &= n (1+y)^{n-1} + (1+y)^n, \label{thm1_cor_id3} \\
\sum_{k=0}^n (-1)^k \left (\sum_{j=0}^k \binom{n}{j} y^j \right ) &= \frac{1}{2}\left ((1-y)^{n} + (-1)^n (1+y)^n \right ), \label{thm1_cor_id4} \\
\sum_{k=0}^n \left (\sum_{j=0}^k \binom{n}{j} x^{k-j} \right ) &= \frac{1}{1-x}\left ( 2^n - x(1+x)^n \right ), \label{thm1_cor_id5} \\
\sum_{k=0}^n \left (\sum_{j=0}^k \binom{n}{j} (-1)^j y^{n+j-k} \right ) &= (1 - y)^{n-1}. \label{thm1_cor_id6}
\end{align}
\end{corollary}
\begin{proof}
These are special cases of \eqref{main_id1}. Set $y=1$ and $y=-1$ to get \eqref{thm1_cor_id1} and \eqref{thm1_cor_id2}, respectively.
Identity \eqref{thm1_cor_id3} is obtained from \eqref{main_id1} by taking the limit $x\to 1$ and using l'Hospital's rule. 
The remaining three follow by setting $x=-1$, $y=1/x$, and $x=-1/y$ in~\eqref{main_id1}, respectively.
\end{proof}

\begin{remark}
Identity \eqref{thm1_cor_id2} is seen to be true by noting that
\begin{equation}\label{GKP_id}
\sum_{j=0}^k \binom{n}{j} (-1)^j = (-1)^k \binom{n-1}{k},
\end{equation}
which is Equation (5.16) in \cite{Graham} or Equation (39) in \cite{Stenlund}. This immediately yields
$$\sum_{k=0}^n x^k \left (\sum_{j=0}^k \binom{n}{j} (-1)^j \right ) = \sum_{k=0}^n \binom{n-1}{k} (-1)^k x^k = (1-x)^{n-1}.$$
\end{remark}

\begin{corollary}
With $x$ and $y$ being complex numbers we have the following relations
\begin{equation}
\sum_{k=1}^n x^k \left (\sum_{j=1}^k \binom{n}{j} \frac{y^j}{j} \right ) 
= \frac{1}{1-x} \sum_{k=1}^n \binom{n}{k} \frac{y^k}{k} (x^k - x^{n+1}) \label{thm1_cor_id7}
\end{equation}
or equivalently
\begin{equation*}
\sum_{k=1}^n x^k \left (\sum_{j=1}^k \binom{n}{j} \frac{y^j}{j} \right ) 
= \frac{1}{1-x} \sum_{k=1}^n \frac{(1+xy)^k - x^{n+1}(1+y)^k}{k} - H_n \frac{1-x^{n+1}}{1-x}, 
\end{equation*}
where $H_n$ is the $n$th harmonic number. In particular,
\begin{equation*}
\sum_{k=1}^n \sum_{j=1}^k \binom{n}{j} \frac{y^j}{j} = (n+1) \sum_{k=1}^n \frac{(1+y)^k}{k} + 1 - (1+y)^n - (n+1)H_n. 
\end{equation*}
\end{corollary}
\begin{proof}
Divide \eqref{main_id1} by $y$ and write the result as
\begin{align*}
\sum_{k=0}^n \frac{x^k}{y} + \sum_{k=1}^n x^k \sum_{j=1}^k \binom{n}{j} y^{j-1} &= \frac{1}{1-x}\Big ( \frac{1}{y} - \frac{x^{n+1}}{y} \\
&\qquad + \sum_{k=1}^n \binom{n}{k} x^k y^{k-1} - x^{n+1}\sum_{k=1}^n \binom{n}{k}y^{k-1} \Big ).
\end{align*}
Integrate with respect to $y$ from $0$ to $z$ and change back the notation to $y$. This yields
\begin{equation*}
\sum_{k=0}^n x^k \ln y + \sum_{k=1}^n x^k \sum_{j=1}^k \binom{n}{j} \frac{y^j}{j} 
= \ln y \frac{1-x^{n+1}}{1-x} + \frac{1}{1-x} \sum_{k=1}^n \binom{n}{k} \frac{y^k}{k} (x^k - x^{n+1}).
\end{equation*}
This completes the proof of the first identity. The second identity is obtained from the first by using \cite[Prop. 1, Eq. (22)]{Adegoke3}
\begin{equation*}
\sum_{k=1}^n \binom{n}{k} \frac{y^k}{k} = \sum_{k=1}^n \frac{(1+y)^k}{k} - H_n
\end{equation*}
and some simplifications. The particular case follows by taking the limit $x\to 1$ and using l'Hospital's rule. 
\end{proof}

Two special double sums evaluations that we want to record are
\begin{equation*}
\sum_{k=1}^n \sum_{j=1}^k \binom{n}{j} \frac{1}{j} = (n+1) \sum_{k=1}^n \frac{2^k}{k} + 1 - 2^n - (n+1)H_n 
\end{equation*}
and
\begin{equation*}
\sum_{k=1}^n \sum_{j=1}^k \binom{n}{j} \frac{(-1)^j}{j} = 1 - (n+1)H_n. 
\end{equation*}

\begin{corollary}
With $x$ and $y$ being complex numbers we have the following identity
\begin{equation}
\sum_{k=1}^n \frac{x^k}{k} \left (\sum_{j=0}^k \binom{n}{j} y^j \right ) 
= (1+y)^n \sum_{m=1}^n \frac{x^m}{m} - \sum_{k=1}^n \binom{n}{k} y^k \sum_{m=1}^{k-1} \frac{x^m}{m}. \label{thm1_cor_id9}
\end{equation}
In particular,
\begin{equation}\label{aqejkrz}
\sum_{k=1}^n \frac{1}{k} \left (\sum_{j=0}^k \binom{n}{j} y^j \right ) = H_n (1+y)^n - \sum_{k=1}^n \binom{n}{k} y^k H_{k-1}. 
\end{equation}
\end{corollary}
\begin{proof}
Divide \eqref{main_id1} by $x$ and write the result as
\begin{align*}
\frac{1}{x} + \sum_{k=1}^n x^{k-1} \sum_{j=0}^k \binom{n}{j} y^{j} = \frac{1}{x(1-x)} 
+ \sum_{k=1}^n \binom{n}{k} y^k \frac{x^{k-1}}{1-x} - \frac{x^{n}}{1-x}(1+y)^n 
\end{align*}
or equivalently
\begin{align*}
-\frac{1}{1-x} + \sum_{k=1}^n x^{k-1} \sum_{j=0}^k \binom{n}{j} y^{j} 
= \sum_{k=1}^n \binom{n}{k} y^k \frac{x^{k-1}}{1-x} - \frac{x^{n}}{1-x}(1+y)^n. 
\end{align*}
Integrate with respect to $x$ from $0$ to $z$ using
\begin{equation*}
\int_0^z \frac{x^n}{1-x} dx = - \sum_{m=1}^n \frac{z^m}{m} - \ln |1-z|, \qquad (n\geq 0).
\end{equation*}
After changing back the notation to $x$ this procedure yields
\begin{align*}
\ln |1-x| + \sum_{k=1}^n \frac{x^k}{k} \sum_{j=0}^k \binom{n}{j} y^j &= 
- \sum_{k=1}^n \binom{n}{k} y^k \left ( \sum_{m=1}^{k-1} \frac{x^m}{m} + \ln |1-x| \right ) \\
&\quad\qquad + (1+y)^n \left ( \sum_{m=1}^{n} \frac{x^m}{m} + \ln |1-x| \right ) \\
&= - \sum_{k=1}^n \binom{n}{k} y^k \sum_{m=1}^{k-1} \frac{x^m}{m} - \ln |1-x| \left ( (1+y)^n -1 \right ) \\
&\quad\qquad + (1+y)^n \sum_{m=1}^{n} \frac{x^m}{m} + \ln |1-x| (1+y)^n.
\end{align*}
The logarithmic terms cancel out and the proof is completed. The particular case follows by setting $x=1$. 
\end{proof}

We record the evaluations 
\begin{equation}\label{eq:evalxy=1OfCor36}
\sum_{k=1}^n \frac{1}{k} \sum_{j=0}^k \binom{n}{j} = 2^n H_n - \sum_{k=1}^n \binom{n}{k} H_{k-1}
\end{equation}
and
\begin{equation*}
\sum_{k=1}^n \frac{1}{k} \sum_{j=0}^k \binom{n}{j} (-1)^j = \frac{1}{n} - H_n. 
\end{equation*}

More basic identities can be derived from \eqref{main_id1}. Mimicking the results obtained above, we have the following.

\begin{corollary}
    With $x$ and $y$ being complex numbers we have the following relation
    \begin{equation*}
        \sum_{k=0}^nx^k\left(\sum_{j=1}^kj\binom{n}{j}y^j\right)=\frac{nxy}{1-x}((1+xy)^{n-1}-x^{n}(1+y)^{n-1}).
    \end{equation*}
    In particular, the following sums hold:
    \begin{align*}
        \sum_{k=0}^nx^k\left(\sum_{j=1}^kj\binom{n}{j}\right)&=\frac{nx}{1-x}((1+x)^{n-1}-2^{n-1}x^{n}),\\
        \sum_{k=0}^nx^k\left(\sum_{j=1}^k(-1)^jj\binom{n}{j}\right)&=-nx(1-x)^{n-2},
    \end{align*}
    where in the last one identity it is assumed that $n\geq 2$.
\end{corollary}
\begin{proof}
Apply the operator $y\cdot \frac{d}{dy}$ to \eqref{main_id1}. Particular cases are $y=1$ and $y=-1$, respectively.
\end{proof}
\begin{remark}
    We note that the second particular identity can be obtained directly via two intermediate steps. First, we used the identity
    $$\sum_{j=1}^k(-1)^j{j\binom nj}=(-1)^kn\binom{n-2}{k-1},$$
    which follows from \eqref{GKP_id} and $j\binom{n}{j}=n\binom{n-1}{j-1}$. Second, we have
    $$\sum_{k=1}^{n-1}{(-1)^kn\binom {n-2}{k-1} x^k}=-nx(1-x)^{n-2}.$$
\end{remark}

It is clear that more identities are possible, as one can use similar substitutions or transformation as in Corollary \ref{cor:3.4}.

\section{Some consequences - sums involving Fibonacci numbers}

We can obtain many new and interesting Fibonacci and also Gibonacci identities from Theorem \ref{main_thm1}. 
For example, we have the following. 

\begin{proposition}
If $m$ and $t$ are integers such that $m\neq 1$, then
\begin{equation*}
\sum_{k=0}^n \sum_{j=0}^k \binom{n}{j} \left (\frac{F_m}{F_{m-1}}\right )^j F_{k+j+t-1} = F_{m-1}^{-n}\, F_{(m+1)n+t+1} 
- \sum_{k=0}^n \binom{n}{k} \left (\frac{F_m}{F_{m-1}}\right )^k F_{2k+t}
\end{equation*}
and
\begin{equation*}
\sum_{k=0}^n \sum_{j=0}^k \binom{n}{j} \left (\frac{F_m}{F_{m-1}}\right )^j L_{k+j+t-1} = F_{m-1}^{-n}\, L_{(m+1)n+t+1} 
- \sum_{k=0}^n \binom{n}{k} \left (\frac{F_m}{F_{m-1}}\right )^k L_{2k+t}.
\end{equation*}
In particular, 
\begin{equation*}
\sum_{k=0}^n \sum_{j=0}^k \binom{n}{j} F_{k+j+t-1} = F_{3n+t+1} - \begin{cases}
5^{n/2} F_{n+t}, & \text{\rm if $n$ is even;} \\ 
5^{(n-1)/2} L_{n+t}, & \text{\rm if $n$ is odd,} 
\end{cases}
\end{equation*}
\begin{equation*}
\sum_{k=0}^n \sum_{j=0}^k \binom{n}{j} L_{k+j+t-1} = L_{3n+t+1} - \begin{cases}
5^{n/2} L_{n+t}, & \text{\rm if $n$ is even;} \\ 
5^{(n+1)/2} F_{n+t}, & \text{\rm if $n$ is odd,} 
\end{cases}
\end{equation*}
and
\begin{equation*}
\sum_{k=0}^n \sum_{j=0}^k \binom{n}{j} (-1)^j F_{k+j+t-1} = (-1)^{n+1} (F_{n+t} - F_{t+1}),
\end{equation*}
\begin{equation*}
\sum_{k=0}^n \sum_{j=0}^k \binom{n}{j} (-1)^j L_{k+j+t-1} = (-1)^{n+1} (L_{n+t} - L_{t+1}).
\end{equation*}
\end{proposition}
\begin{proof}
Set $x=\alpha$ and $x=\beta$, in turn, in \eqref{main_id1} to get
\begin{align}
\sum_{k=0}^n \alpha^{k-1} \sum_{j=0}^k \binom{n}{j} y^j &= \alpha^{n+1} (1+y)^n - (1+\alpha y)^n,\label{l5jcxzd} \\
\sum_{k=0}^n \beta^{k-1} \sum_{j=0}^k \binom{n}{j} y^j &= \beta^{n+1} (1+y)^n - (1+\beta y)^n\label{t561avr}.
\end{align}
Now, work with the basic relations  
\begin{equation*}
\alpha^m = \alpha F_m + F_{m-1} \qquad \text{and} \qquad \beta^m = \beta F_m + F_{m-1}.
\end{equation*}
They yield together with $y=\tfrac{F_{m}}{F_{m-1}}\alpha$ and $y=\tfrac{F_{m}}{F_{m-1}}\beta$, in turn, the following
\begin{align*}
\sum_{k=0}^n \alpha^{k-1+t} \sum_{j=0}^k \binom{n}{j} \left (\frac{F_m}{F_{m-1}}\right )^j \alpha^j &= 
F_{m-1}^{-n} \alpha^{(m+1)n+t+1} - \alpha^t \left (1+\frac{F_m}{F_{m-1}}\alpha^2 \right), \\
\sum_{k=0}^n \beta^{k-1+t} \sum_{j=0}^k \binom{n}{j} \left (\frac{F_m}{F_{m-1}}\right )^j \beta^j &= 
F_{m-1}^{-n} \beta^{(m+1)n+t+1} - \beta^t \left (1+\frac{F_m}{F_{m-1}}\beta^2 \right).
\end{align*}
The results follow by combining the above sums according to the Binet forms. The special cases 
correspond to $m=2$ and $m=-1$, respectively. The Fibonacci sums on the right are elementary. 
They can be found, for instance, in \cite{CarFer} and \cite{Layman}.
\end{proof}

\begin{proposition}
If $m$ and $t$ are integers such that $m\neq 2$, then
\begin{equation*}
\sum_{k=0}^n \sum_{j=0}^k \binom{n}{j} F_{m-2}^{-j} F_{2k+mj+t+1} = \left (\frac{F_m}{F_{m-2}}\right )^n F_{4n+t+2} 
- \sum_{k=0}^n \binom{n}{k} F_{m-2}^{-k} F_{(m+2)k+t}
\end{equation*}
and
\begin{equation*}
\sum_{k=0}^n \sum_{j=0}^k \binom{n}{j} F_{m-2}^{-j} L_{2k+mj+t+1} = \left (\frac{F_m}{F_{m-2}}\right )^n L_{4n+t+2} 
- \sum_{k=0}^n \binom{n}{k} F_{m-2}^{-k} L_{(m+2)k+t}.
\end{equation*}
In particular, 
\begin{equation*}
\sum_{k=0}^n \sum_{j=0}^k \binom{n}{j} F_{2k+j+t+1} = F_{4n+t+2} - 2^n F_{2n+t},
\end{equation*}
\begin{equation*}
\sum_{k=0}^n \sum_{j=0}^k \binom{n}{j} L_{2k+j+t+1} = L_{4n+t+2} - 2^n L_{2n+t},
\end{equation*}
and
\begin{equation*}
\sum_{k=0}^n \sum_{j=0}^k \binom{n}{j} (-1)^j 3^{-j} F_{2k-2j+t+1} = 3^{-n} (F_{4n+t+2} - 2^n F_{t}),
\end{equation*}
\begin{equation*}
\sum_{k=0}^n \sum_{j=0}^k \binom{n}{j} (-1)^j 3^{-j} L_{2k-2j+t+1} = 3^{-n} (L_{4n+t+2} - 2^n L_{t}).
\end{equation*}
\end{proposition}
\begin{proof}
The proof is very similar to the previous proof. Set $x=\alpha^2$ and $x=\beta^2$, in turn, in \eqref{main_id1}. Now, work with the less popular but also obvious relations  
\begin{equation*}
\alpha^m = \alpha^2 F_m - F_{m-2} \qquad \text{and} \qquad \beta^m = \beta^2 F_m - F_{m-2}.
\end{equation*}
We omit the rest as it is clear how to finish the proof. The special cases correspond to $m=1$ and $m=-2$, respectively.
\end{proof}

Working with $x=\alpha^3$ and $x=\beta^3$ in \eqref{main_id1} and making use of 
$$1-\alpha^3 = -2\alpha \quad \text{and} \quad 1-\beta^3 = -2\beta$$
we get the next result without efforts.

\begin{proposition}
If $m$ and $t$ are integers such that $m\neq 2$, then
\begin{equation*}
\sum_{k=0}^n \sum_{j=0}^k \binom{n}{j} F_{m-2}^{-j} F_{3k+mj+t+1} = \frac{1}{2}\left (\frac{F_m}{F_{m-2}}\right )^n F_{5n+t+3} 
- \frac{1}{2} \sum_{k=0}^n \binom{n}{k} F_{m-2}^{-k} F_{(m+3)k+t}
\end{equation*}
and
\begin{equation*}
\sum_{k=0}^n \sum_{j=0}^k \binom{n}{j} F_{m-2}^{-j} L_{3k+mj+t+1} = \frac{1}{2} \left (\frac{F_m}{F_{m-2}}\right )^n L_{5n+t+3} 
- \frac{1}{2} \sum_{k=0}^n \binom{n}{k} F_{m-2}^{-k} L_{(m+3)k+t},
\end{equation*}
of which
\begin{equation*}
\sum_{k=0}^n \sum_{j=0}^k \binom{n}{j} F_{3k+j+t+1} = \frac{1}{2} \left ( F_{5n+t+3} - 3^n F_{2n+t}\right )
\end{equation*}
and 
\begin{equation*}
\sum_{k=0}^n \sum_{j=0}^k \binom{n}{j} L_{3k+j+t+1} = \frac{1}{2} \left ( L_{5n+t+3} - 3^n L_{2n+t}\right )
\end{equation*}
are special cases.
\end{proposition}

More sums of this nature are possible. We state two additional identities as showcases:
\begin{equation*}
  \sum_{k=0}^n (-1)^k \sum_{j=0}^k \binom{n}{j} F_{3j+3k+t} = (-1)^n (2^{n-1} F_{5n+t+1} + 2^{2n-1} F_{3n+t-2}),
\end{equation*}
\begin{equation*}
  \sum_{k=0}^n \sum_{j=0}^k (-1)^j \binom{n}{j} F_{3j+3k+t} = (-1)^n (2^{n-1} F_{4n+t+2} - 2^{2n-1} F_{3n+t-1}).
\end{equation*}

We can also derive slightly different but very general results involving four parameters. First, recall the following.

\begin{lemma}\label{lemma:powersofalpha}
The following identities hold:
\begin{align*}
   \alpha^{4m} + 1 &= L_{2m}\alpha^{2m},\\
   \alpha^{4m+2} - 1 &= L_{2m+1}\alpha^{2m+1},\\
   \alpha^{4m+2} + 1 &= F_{2m+1}(\alpha^{2m}+\alpha^{2m+2})
\end{align*}
and similar identities hold for $\beta$.
\end{lemma}
\begin{proof}
The first and the second are direct consequences of the simple relation
\begin{equation}\label{powers}
(-1)^s + \alpha^{2s} = \alpha^s L_s.
\end{equation}  
For the third item, use $\alpha^u=F_u\alpha+F_{u-1}$ in combination with $F_u+F_{u+2}=L_{u+1}$, $F_uL_u=F_{2u}$ and $F_{u+1}L_u=F_{2u+1}+1$.
\end{proof}

The next result is proved in \cite{Kilic}.
\begin{lemma}\label{lem_kilic}
For integers $r$ and $s$ we have
\begin{equation}\label{kilic1}
\sum_{k=0}^n \binom{n}{k} F_{2sk+r} = \begin{cases}
5^{(n-1)/2}\,F_s^n L_{sn+r}, & \text{\rm if $n$ is odd and $s$ is odd;} \\ 
5^{n/2}\,F_s^n F_{sn+r}, & \text{\rm if $n$ is even and $s$ is odd;} \\
L_s^n F_{sn+r}, & \text{\rm if $s$ is even;}
\end{cases}
\end{equation}	
and
\begin{equation*}
\sum_{k=0}^n \binom{n}{k} L_{2sk+r} = \begin{cases}
5^{(n+1)/2}\,F_s^n L_{sn+r}, & \text{\rm if $n$ is odd and $s$ is odd;} \\ 
5^{n/2}\,F_s^n L_{sn+r}, & \text{\rm if $n$ is even and $s$ is odd;} \\
L_s^n L_{sn+r}, & \text{\rm if $s$ is even.}
\end{cases}
\end{equation*}
\end{lemma}

\begin{proposition}
If $u,v$ and $t$ are integers, then
\begin{align*}
&\sum_{k=0}^n \sum_{j=0}^k \binom{n}{j} F_{(4u+2)k+(4v+2)j+t} \nonumber \\
&\quad = \frac{F_{2v+1}^n}{L_{2u+1}} \begin{cases}
5^{n/2}\,F_{(4u+2v+3)n+2u+1+t}, & \text{\rm if $n$ is even;} \\ 
5^{(n-1)/2}\,L_{(4u+2v+3)n+2u+1+t}, & \text{\rm if $n$ is odd;}
\end{cases} 
- \frac{L_{2u+2v+2}^n}{L_{2u+1}} F_{(2u+2v+2)n-(2u+1)+t}
\end{align*}
and
\begin{align*}
&\sum_{k=0}^n \sum_{j=0}^k \binom{n}{j} L_{(4u+2)k+(4v+2)j+t} \nonumber \\
&\quad = \frac{F_{2v+1}^n}{L_{2u+1}} \begin{cases}
5^{n/2}\,L_{(4u+2v+3)n+2u+1+t}, & \text{\rm if $n$ is even;} \\ 
5^{(n+1)/2}\,L_{(4u+2v+3)n+2u+1+t}, & \text{\rm if $n$ is odd;}
\end{cases} 
- \frac{L_{2u+2v+2}^n}{L_{2u+1}} L_{(2u+2v+2)n-(2u+1)+t}
\end{align*}
In particular,
\begin{align*}
\sum_{k=0}^n \sum_{j=0}^k \binom{n}{j} F_{2(k+j)+t}& = \begin{cases}
5^{n/2}\,F_{3n+1+t}, & \text{\rm if $n$ is even;} \\ 
5^{(n-1)/2}\,L_{3n+1+t}, & \text{\rm if $n$ is odd;}
\end{cases} 
- 3^n F_{2n-1+t}\\
\sum_{k=0}^n \sum_{j=0}^k \binom{n}{j} L_{2(k+j)+t}& = \begin{cases}
5^{n/2}\,L_{3n+1+t}, & \text{\rm if $n$ is even;} \\ 
5^{(n+1)/2}\,L_{3n+1+t}, & \text{\rm if $n$ is odd;}
\end{cases} 
- 3^n L_{2n-1+t}
\end{align*}
and
\begin{align*}
\sum_{k=0}^n \sum_{j=0}^k \binom{n}{j} F_{6(k+j)+t}& = 2^{n-2} \left ( \begin{cases}
5^{n/2}\,F_{3(3n+1)+t}, & \text{\rm if $n$ is even;} \\ 
5^{(n-1)/2}\,L_{3(3n+1)+t}, & \text{\rm if $n$ is odd;}
\end{cases} 
- 3^{2n} F_{3(2n-1)+t} \right ),\\
\sum_{k=0}^n \sum_{j=0}^k \binom{n}{j} L_{6(k+j)+t} &= 2^{n-2} \left ( \begin{cases}
5^{n/2}\,L_{3(3n+1)+t}, & \text{\rm if $n$ is even;} \\ 
5^{(n+1)/2}\,L_{3(3n+1)+t}, & \text{\rm if $n$ is odd;}
\end{cases} 
- 3^{2n} L_{3(2n-1)+t} \right ).
\end{align*}
\end{proposition}
\begin{proof}
 Set $x=\alpha^{4u+2}$ and $y=\alpha^{4v+2}$ in \eqref{main_id1}. This gives
 $$\sum_{k=0}^n\sum_{j=0}^n\binom{n}{j}\alpha^{4uk+4vj+2k+2j}=\frac{1}{1-\alpha^{4u+2}}\left((1+\alpha^{4j+4v+4})^n-(\alpha^{4u+2})^{n+1}(1+\alpha^{4v+2})^n\right).$$
 We now simplify using Lemma \ref{lemma:powersofalpha}, which gives
 $$\sum_{k=0}^n\sum_{j=0}^n\binom{n}{j}\alpha^{4uk+4vj+2k+2j}=\frac{F_{2v+1}^n}{L_{2u+1}}\cdot \frac{(\alpha^{2v}+\alpha^{2v+2})^n}{\alpha^{2u+1}}\alpha^{(4u+2)(n+1)}-\frac{L_{2u+2v+2}^n}{L_{2u+1}}\cdot \frac{\alpha^{2un+2vn+2n}}{\alpha^{2u+1}}.$$

 Expressing
 $$(\alpha^{2v}+\alpha^{2v+2})^n=\alpha^{2vn}\sum_{k=0}^n\binom{n}{k}\alpha^{2k}=\sum_{k=0}^n\binom{n}{k}\alpha^{2vn+2k}$$
 yields the identity
 \begin{multline*}
 \sum_{k=0}^n\sum_{j=0}^k\binom{n}{j} L_{2u+1}F_{k(4u+2)+j(4v+2)+t}\\
 = F^n_{2v+1}\sum_{k=0}^n\binom{n}{k} F_{n(4u+2v+2)+2k+2u+1+t}-L_{2u+2v+2}^n F_{n(2u+2v+2)-2u-1+t}.
 \end{multline*}
The final identity follows since we have
\begin{equation*}
\sum_{k=0}^n \binom{n}{k} F_{An+B+2k} = \begin{cases}
5^{n/2}\,F_{(A+1)n+B}, & \text{\rm if $n$ is even;} \\ 
5^{(n-1)/2}\,L_{(A+1)n+B}, & \text{\rm if $n$ is odd,} 
\end{cases}
\end{equation*}		
which is a special case of \eqref{kilic1}.
\end{proof}

\begin{proposition}
If $u,v$ and $t$ are integers such that $u$ is even and $v$ is odd, then
\begin{align*}
&\sum_{k=0}^n \sum_{j=0}^k \binom{n}{j} F_{v(2k+1)+2uj+t} \nonumber \\
&\quad = 
\frac{L_{u}^n}{L_{v}} F_{(2v+u)n+2v+t} - \frac{1}{L_{v}} \begin{cases}
5^{n/2}\,F_{v+u}^n F_{(v+u)n+t}, & \text{\rm if $n$ is even;} \\ 
5^{(n-1)/2}\,F_{v+u}^n L_{(v+u)n+t}, & \text{\rm if $n$ is odd;}
\end{cases} 
\end{align*}
and
\begin{align*}
&\sum_{k=0}^n \sum_{j=0}^k \binom{n}{j} L_{v(2k+1)+2uj+t} \nonumber \\
&\quad = 
\frac{L_{u}^n}{L_{v}} L_{(2v+u)n+2v+t} - \frac{1}{L_{v}} \begin{cases}
5^{n/2}\,F_{v+u}^n L_{(v+u)n+t}, & \text{\rm if $n$ is even;} \\ 
5^{(n+1)/2}\,F_{v+u}^n F_{(v+u)n+t}, & \text{\rm if $n$ is odd;}
\end{cases} 
\end{align*}
\end{proposition}
\begin{proof}
Let $u$ be even and $v$ be odd. Then from \eqref{powers} with $x=\alpha^{2v}$ and $y=\alpha^{2u}$ we get from the main identity \eqref{main_id1}
\begin{equation*}
\sum_{k=0}^n \sum_{j=0}^k \binom{n}{j} \alpha^{v(2k+1)+2uj+t} = \frac{1}{L_v}\left (L_u^n \alpha^{(2v+u)n+2v+t} 
- \alpha^t \left (1+\alpha^{2(v+u)}\right )^n \right ).
\end{equation*}
Similarly, with $x=\beta^{2v}$ and $y=\beta^{2u}$
\begin{equation*}
\sum_{k=0}^n \sum_{j=0}^k \binom{n}{j} \beta^{v(2k+1)+2uj+t} = \frac{1}{L_v}\left (L_u^n \beta^{(2v+u)n+2v+t} 
- \beta^t \left (1+\beta^{2(v+u)}\right )^n \right ).
\end{equation*}
Upon combining according to the Binet form we end up with
\begin{align*}
\sum_{k=0}^n \sum_{j=0}^k \binom{n}{j} F_{v(2k+1)+2uj+t} &= \frac{L_{u}^n}{L_{v}} F_{(2v+u)n+2v+t} - \frac{1}{L_{v}} \sum_{k=0}^n \binom{n}{k} F_{2(v+u)k+t} \\
\sum_{k=0}^n \sum_{j=0}^k \binom{n}{j} L_{v(2k+1)+2uj+t} &= \frac{L_{u}^n}{L_{v}} L_{(2v+u)n+2v+t} - \frac{1}{L_{v}} \sum_{k=0}^n \binom{n}{k} L_{2(v+u)k+t}.
\end{align*}
The final identities follow from Lemma \ref{lem_kilic} using the fact that $v+u$ is odd.
\end{proof}

\begin{corollary}
The following identities hold:
\begin{equation*}
\sum_{k=0}^n F_{2k+1} \sum_{j=0}^k \binom{n}{j} = 2^n F_{2n+2} - \begin{cases}
5^{n/2}\,F_{n}, & \text{\rm if $n$ is even;} \\ 
5^{(n-1)/2}\,L_{n}, & \text{\rm if $n$ is odd;}
\end{cases} 
\end{equation*}
\begin{equation*}
\sum_{k=0}^n L_{2k+1} \sum_{j=0}^k \binom{n}{j} = 2^n L_{2n+2} - \begin{cases}
5^{n/2}\,L_{n}, & \text{\rm if $n$ is even;} \\ 
5^{(n+1)/2}\,F_{n}, & \text{\rm if $n$ is odd;}
\end{cases} 
\end{equation*}
\begin{equation*}
\sum_{k=0}^n \sum_{j=0}^k \binom{n}{j} F_{2k+1+4j} = 3^n F_{4n+2} - \begin{cases}
5^{n/2}\,2^n \,F_{3n}, & \text{\rm if $n$ is even;} \\ 
5^{(n-1)/2}\,2^n\,L_{3n}, & \text{\rm if $n$ is odd;}
\end{cases} 
\end{equation*}
and
\begin{equation*}
\sum_{k=0}^n \sum_{j=0}^k \binom{n}{j} L_{2k+1+4j} = 3^n L_{4n+2} - \begin{cases}
5^{n/2}\,2^n \,L_{3n}, & \text{\rm if $n$ is even;} \\ 
5^{(n+1)/2}\,2^n\,F_{3n}, & \text{\rm if $n$ is odd.}
\end{cases} 
\end{equation*}
\end{corollary}

\begin{proposition}
The following identities hold:
    \begin{align*}
        \sum_{k=0}^n\sum_{j=0}^k \binom{n}{j} 2^k F_{j+t} &= 2^{n+1}F_{2n+t}-F_{3n+t},\\
        \sum_{k=0}^n\sum_{j=0}^k \binom{n}{j} 2^k L_{j+t} &= 2^{n+1}L_{2n+t}-L_{3n+t},\\
        \sum_{k=0}^n\sum_{j=0}^k (-1)^j \binom{n}{j} 2^k F_{2j+t} &= (-1)^n(2^{n+1}F_{n+t}-F_{3n+t}),\\
        \sum_{k=0}^n\sum_{j=0}^k (-1)^j \binom{n}{j} 2^k L_{2j+t} &= (-1)^n(2^{n+1}L_{n+t}-L_{3n+t}).
    \end{align*}
\end{proposition}
\begin{proof}
Apply \eqref{main_id1} with $x=2$ and $y=\alpha$ ($y=-\alpha^2$, respectively) and use $1+2\alpha=\alpha^3$ ($1-2\alpha^2=-\alpha^3$, respectively). Then, in turn, do the same for $\beta$ and combine using the Binet form.
\end{proof}

We can make things a bit more general. Combining~\eqref{l5jcxzd} and~\eqref{t561avr}, we have the identity stated in Lemma~\ref{lem.x7w3t2c}.

\begin{lemma}\label{lem.x7w3t2c}
If $n$, $r$ and $s$ are non-negative integers, $t$ is an integer and $x$ and $y$ are complex numbers, then
\begin{align}\label{jvq52k6}
&\sum_{k = 0}^n G_{k + t} \sum_{j = r}^k {\binom{n-r-s}{j-r}x^{n - j - s}y^{j - r}} = G_{n + t + 2} \left( {x + y} \right)^{n-r-s} \nonumber\\
&\qquad - \left( A\alpha ^{t + r + 1} \left( {x + \alpha y} \right)^{n-r-s} + B\beta ^{t + r + 1} \left( {x + \beta y} \right)^{n-r-s}\right).
\end{align}
\end{lemma}

\begin{proposition}
If $n$, $r$ and $s$ are non-negative integers and $m$ and $t$ are integers, then
\begin{align}\label{rzl5whv}
&\sum_{k = 0}^n G_{k + t} \sum_{j = r}^k {( - 1)^j \binom{{n - r - s}}{{j - r}}F_{m + 1}^{n - j - s} F_m^{j - r} } \nonumber\\
&\qquad = ( - 1)^{m(n-r-s) - r + 1} G_{t + r + 1 - m(n-r-s)} + ( - 1)^r F_{m - 1}^{n - r - s} G_{n + t + 2}.
\end{align}
In particular,
\begin{equation}\label{ehr75lp}
\sum_{k = r}^n ( - 1)^k \binom{{n - r - s}}{{k - r}}G_{k + t} = ( - 1)^{n - s} G_{t + 2r + s - n} 
\end{equation}
and
\begin{equation}
\sum_{k = 0}^n G_{k + t} \sum_{j = r}^k {\binom{{n - r - s}}{{j - r}}} = - G_{t + r + 1 + 2(n - r - s)} + 2^{n - r - s} G_{n + t + 2}.
\end{equation}
\end{proposition}
\begin{proof}
Set $x=F_{m+1}$ and $y=-F_m$ in~\eqref{jvq52k6} and use
\begin{equation*}
F_{m+1} - F_m\alpha = \beta^m \quad\text{and}\quad F_{m+1} - F_m\beta = \alpha^m.
\end{equation*}
The particular cases are evaluations of~\eqref{rzl5whv} at $m=1$ and $m=-2$, respectively. Note that, in obtaining~\eqref{ehr75lp}, we used
\begin{equation*}
\sum_{j = r}^k ( - 1)^j \binom{{n - r - s}}{{j - r}} = ( - 1)^k \binom{{n - r - s - 1}}{{k - r}},
\end{equation*}
which is a consequence of the binomial theorem and of which~\eqref{GKP_id} is a particular case.
\end{proof}

\begin{lemma}\label{lem.hg1ierv}
If $n$, $r$ and $s$ are non-negative integers, $t$ is an integer and $x$ and $y$ are complex numbers, then
\begin{align}\label{jvq52k7}
&\sum_{k = 0}^n G_{2k + t} \sum_{j = r}^k {\binom{n-r-s}{j-r}x^{n - j - s}y^{j - r}} = G_{2n + t + 1} \left( {x + y} \right)^{n-r-s} \nonumber\\
&\qquad - \left( A\alpha^{t + 2r - 1} (x + \alpha^2 y)^{n-r-s} + B\beta^{t + 2r - 1} (x + \beta^2 y)^{n-r-s}\right).
\end{align}
\end{lemma}

\begin{proposition}
If $n$, $r$ and $s$ are non-negative integers and $m$ and $t$ are integers, then
\begin{align*}
&\sum_{k = 0}^n {G_{2k + t} \sum_{j = r}^k {( - 1)^j \binom{{n - r - s}}{{j - r}}F_{m + 2}^{n - j - s} F_m^{j - r} } }\nonumber\\
&\qquad  = ( - 1)^{m(n - r - s) + r - 1} G_{t + 2r - 1 - m(n - r - s)}  + ( - 1)^r G_{2n + t + 1} F_{m + 1}^{n - r - s}.
\end{align*}
In particular,
\begin{equation*}
\sum_{k = r}^n ( - 1)^k \binom{{n - r - s}}{{k - r}}G_{2k + t} = ( - 1)^{n + s} G_{t + n - s + 1}.
\end{equation*}
\end{proposition}
\begin{proof}
Set $x=F_{m+2}$ and $y=-F_m$ in~\eqref{jvq52k7} and use
\begin{equation*}
F_{m+2} - F_m\alpha^2 = \beta^m \quad\text{and}\quad F_{m+2} - F_m\beta = \alpha^m.
\end{equation*}
\end{proof}

\begin{proposition}
If $n$, $r$ and $s$ are non-negative integers and $t$ is an integer, then
\begin{align*}
&\sum_{k = 0}^n {G_{2k + t} \sum_{j = r}^k {\binom{{n - r - s}}{{j - r}} } }\nonumber\\
&\qquad  =2^{n - r - s}\,G_{2n + t + 1} \nonumber\\
&\qquad\qquad - 
\begin{cases}
 5^{(n - r - s)/2}\, G_{t + n + r - s - 1},&\text{if $n+r+s$ is even;}  \\ 
 5^{(n - r - s - 1)/2} \left( {G_{t + n + r - s}  + G_{t + n + r - s -2 } } \right),&\text{if $n+r+s$ is odd.} \\ 
 \end{cases} 
\end{align*}
\end{proposition}
\begin{proof}
Set $x=1=y$ in~\eqref{jvq52k6}.
\end{proof}

Setting $x=\alpha^3$ and $x=\beta^3$, in turn, in~\eqref{main_id1} and combining according to the Binet formula gives the identity stated in Lemma~\ref{lem.tk9vi87}.
\begin{lemma}\label{lem.tk9vi87}
If $n$, $r$ and $s$ are non-negative integers, $t$ is an integer and $x$ and $y$ are complex numbers, then
\begin{align}\label{bh927w7}
&\sum_{k = 0}^n G_{3k + t} \sum_{j = r}^k {\binom{n-r-s}{j-r}x^{n - j - s}y^{j - r}} = \frac{1}{2} G_{3n + t + 2} (x + y)^{n-r-s} \nonumber\\
&\qquad - \frac{1}{2} \left( A\alpha^{t + 3r - 1} (x + \alpha^3 y)^{n-r-s} + B\beta^{t + 3r - 1} (x + \beta^3 y)^{n-r-s}\right).
\end{align}
\end{lemma}

\begin{proposition}
If $n$, $r$ and $s$ are non-negative integers and $m$ and $t$ are integers, then
\begin{align*}
&\sum_{k = 0}^n {G_{3k + t} \sum_{j = r}^k {( - 1)^j \binom{{n - r - s}}{{j - r}}F_{m + 3}^{n - j - s} F_m^{j - r} } } \nonumber\\
&\qquad = ( - 1)^{m(n-r-s) - r + 1} 2^{n-r-s-1}G_{t + 3r - 1 - m(n-r-s)}  + ( - 1)^r 2^{n-r-s-1}F_{m + 1}^{n - r - s}  G_{3n + t + 2} .
\end{align*}
In particular,
\begin{equation*}
\sum_{k = 0}^n G_{3k + t} \sum_{j = 0}^k {\binom{{n}}{j}} = 2^{n - 1} \left( {G_{3n + t + 2} - G_{2n + t - 1} } \right),
\end{equation*}
\begin{equation*}
\sum_{k = r}^n ( - 1)^k \binom{{n - r - s}}{{k - r}}G_{3k + t} = ( - 1)^{n - s} 2^{n - r - s} G_{t + 2r + n - s},
\end{equation*}
and
\begin{equation*}
\sum_{k = 0}^n G_{3k + t} \sum_{j = r}^k {\binom{{n - r - s}}{{j - r}}} = 2^{n - r - s - 1} \left( G_{3n + t + 2} - G_{t + 2(n - s) + r - 1} \right).
\end{equation*}
\end{proposition}
\begin{proof}
Set $x=F_{m+3}$ and $y=-F_m$ in~\eqref{bh927w7} and use
\begin{equation*}
F_{m+3} - F_m\alpha^3 = 2\beta^m \quad\text{and}\quad F_{m+3} - F_m\beta^3 = 2\alpha^m.
\end{equation*}
\end{proof}

\begin{proposition}
If $n$, $r$ and $s$ are non-negative integers and $t$ is an integer, then
\begin{align*}
&\sum_{k = 0}^n {G_{3k + t} \sum_{j = r}^k {\binom{{n - r - s}}{{j - r}}2^j } }\nonumber\\
&\qquad  =2^{r-1}3^{n - r - s}\,G_{3n + t + 2} \nonumber\\
&\qquad\qquad - 2^{r-1} 
\begin{cases}
 5^{(n - r - s)/2}\, G_{t + 3(n - s) - 1},&\text{if $n+r+s$ is even;}  \\ 
 5^{(n - r - s - 1)/2} \left( {G_{t + 3(n - s)}  + G_{t + 3(n - s) - 2} } \right),&\text{if $n+r+s$ is odd.} \\ 
 \end{cases} 
\end{align*}
\end{proposition}
\begin{proof}
Use $x=1$ and $y=2$ in~\eqref{bh927w7} and use the fact that
\begin{equation*}
1 + 2\alpha^3 = \alpha^3\sqrt 5 \quad\text{and}\quad 1 + 2\beta^3 = -\beta^3\sqrt 5.
\end{equation*}
\end{proof}

\subsection{Results for $m$-step numbers}

Recall that Fibonacci $m$-step numbers $F^{(m)}_n$ are defined via recurrence
$$F_{n}^{(m)}=F_{n-1}^{(m)}+\cdots+F_{n-m+1}^{(m)}$$
with initial values $F_1^{(m)}=1$ and $F_{j}^{(m)}=0$ for $j=-m+2,\ldots, 0$. We let $T_n=F_n^{(3)}$ and $Q_n=F_n^{(4)}$ denote the Tribonacci and the Tetranacci numbers. Similarly, we define the Lucas $m$-step numbers $L^{(m)}_n$ via the same recurrence, but with different initial values $L_0^{(m)}=m$, $L_1^{(m)}=1$ and $L_j^{(m)}=0$ for $j=-m+2,\ldots, -1$. Both sequences admit Binet-type formula: if (with slight abuse of notation) $\alpha_1,\ldots,\alpha_m$ are the roots of the characteristic equation
\begin{equation}\label{eq:characericsti_eq_m-step}
    x^m=x^{m-1}+\ldots+1,
\end{equation}
then
$$L_n^{(m)}=\alpha_1^{n}+\cdots+\alpha_m^n,\qquad \text{and}\qquad F_n^{(m)}=\sum_{k=1}^n \frac{\alpha_k^{n+1}}{\prod_{j\neq k}(\alpha_k-\alpha_j)}.$$

\begin{lemma}\label{lemma:alpha_fibonacci}
    Fix $m>2$. If $\lambda$ denotes any root of the characteristic equation \eqref{eq:characericsti_eq_m-step}, then $\lambda^{m+1}+1=2\lambda^m$ and $1-\lambda^{m+1}=2\lambda^m(1-\lambda)$
\end{lemma}
\begin{proof}
    Simple use of the characteristic equation.
\end{proof}

With the lemma in mind, we can formulate several simple consequences.
\begin{corollary}
    The following identities hold:
    $$\sum_{k=0}^n\sum_{j=0}^k\binom{n}{j}F^{(m)}_{(m+1)j+t}=n2^{n-1}F^{(m)}_{m(n-1)+t}+2^nF^{(m)}_{mn+t}$$
    and
    $$\sum_{k=0}^n\sum_{j=0}^k\binom{n}{j}L^{(m)}_{(m+1)j+t}=n2^{n-1}L^{(m)}_{m(n-1)+t}+2^nL^{(m)}_{mn+t}$$
\end{corollary}
\begin{proof}
    Set $y=\lambda^{m+1}$ in \eqref{thm1_cor_id3}. This gives, with help of Lemma \ref{lemma:alpha_fibonacci},
    $$\sum_{k=0}^n \sum_{j=0}^k\binom{n}{j}\lambda^{(m+1)j}=n2^{n-1}\lambda^{m(n-1)}+2^n\lambda^{mn}.$$
    Multiply by $\lambda^t$ and use the Binet formula.
\end{proof}

In a similar fashion, we have the next result.

\begin{corollary}
    The following identities hold:
    \begin{align*}
        \sum_{k=0}^n\sum_{j=0}^k(-1)^k\binom{n}{j}F_{(m+1)j+t}^{(m)}&=\frac{1}{2}\left(\sum_{k=0}^n(-1)^k\binom{n}{k}F^{(m)}_{(m+1)k+t}+(-2)^nF_{mn+t}^{(m)}\right),\\
        \sum_{k=0}^n\sum_{j=0}^k(-1)^{k}\binom{n}{j}L_{(m+1)j+t}^{(m)}&=\frac{1}{2}\left(\sum_{k=0}^n(-1)^k\binom{n}{k}L^{(m)}_{(m+1)k+t}+(-2)^nL_{mn+t}^{(m)}\right).
    \end{align*}
\end{corollary}
\begin{proof}
    Set $y=-\lambda^{m+1}$ and proceed as in the previous result.
\end{proof}
We note particular cases of the results (set $m=3$ and $m=4$):
\begin{align*}
    \sum_{k=0}^n\sum_{j=0}^k\binom{n}{j}T_{4j+t}&=n2^{n-1}T_{3(n-1)+t}+2^nT_{3n+t},\\
    \sum_{k=0}^n\sum_{j=0}^k\binom{n}{j}Q_{5j+t}&=n2^{n-1}Q_{4(n-1)+t}+2^nQ_{4n+t},\\
        \sum_{k=0}^n\sum_{j=0}^k(-1)^k\binom{n}{j}T_{4j+t}&=\frac{1}{2}\left(\sum_{k=0}^n(-1)^k\binom{n}{k}T_{4k+t}+(-2)^nT_{3n+t}\right),\\
        \sum_{k=0}^n\sum_{j=0}^k(-1)^k\binom{n}{j}Q_{5j+t}&=\frac{1}{2}\left(\sum_{k=0}^n(-1)^k\binom{n}{k}Q_{5k+t}+(-2)^nQ_{4n+t}\right).
    \end{align*}

\section{Consequences - identities with harmonic numbers}

We employ the technique described in \cite{Adegoke2} to change the polynomial identity into an identity with harmonic numbers.

\begin{proposition}
Let $r,s\in\mathbb{C}\setminus\mathbb{Z}^{-}$ are such that $s\neq 0$ and $r-s\notin \mathbb{Z}^{-}$. Then the following identities hold:
     \begin{align}
        \label{eq:harmonic_diff1b}
        \sum_{k=0}^n\sum_{j=0}^k(-1)^j\frac{\binom{n}{j}}{(j+s)\binom{j+r}{j+s}}&=\frac{n}{s\binom{n+r-1}{s}}+\frac{1}{s\binom{n+r}{s}},\\
        \label{eq:harmonic:diff1c} \sum_{k=0}^n\sum_{j=0}^k (-1)^k\frac{\binom{n}{j}}{(j+s)\binom{j+r}{j+s}}&=\frac{1}{2s\binom{n+r}{s}}+\frac{1}{2}(-1)^n\sum_{k=0}^n\frac{\binom{n}{k}}{(k+s)\binom{k+r}{k+s}}.
     \end{align}
     Furthermore, we also have
    \begin{align}
        \label{eq:harmonic_diff2b}
        \sum_{k=0}^n\sum_{j=0}^k\binom{n}{j}(-1)^j\frac{H_{r-s}-H_{j+r}}{(j+s)\binom{j+r}{j+s}}&=\frac{n}{s}\cdot\frac{H_{n-1+r-s}-H_{n-1+r}}{\binom{n-1+r}{s}}+\frac{H_{n+r-s}-H_{n+r}}{s\binom{n+r}{s}},\\
         \label{eq:harmonic:diff2c} \sum_{k=0}^n\sum_{j=0}^k\binom{n}{j}(-1)^k\frac{H_{r-s}-H_{j+r}}{(j+s)\binom{j+r}{j+s}}&=\frac{H_{n+r-s}-H_{n+r}}{2s\binom{n+r}{s}}+\frac{1}{2}(-1)^n\sum_{k=0}^n\frac{\binom{n}{k}(H_{r-s}-H_{k+r})}{(k+s)\binom{k+r}{k+s}}
    \end{align}
\end{proposition}
\begin{proof}
First, we prove \eqref{eq:harmonic_diff1b}. Replace $y$ with $-y$ in \eqref{thm1_cor_id3} and multiply the result by $(1-y)^ry^{s-1}$. Now, integrate from $0$ to $1$ with respect to $y$ to obtain
$$\int_0^1\sum_{k=0}^n \left (\sum_{j=0}^k \binom{n}{j} (-1)^j(1-y)^ry^{j+s-1}) \right )dy = n \int_0^1(1-y)^{n-1+r}y^{s-1}dy + \int_0^1(1-y)^{n+r}y^{s-1}dy. $$
    We now use the following representation of the Beta integral:
    $$\int_0^1 x^u(1-x)^vdx=\frac{1}{(u+1)\binom{u+v+1}{u+1}},$$
    which is valid for $\Re (u)>-1$ and $\Re (v)>-1$. This gives, after substitution $r\mapsto r-s$, the sum
    \eqref{eq:harmonic_diff1b}. To obtain \eqref{eq:harmonic_diff2b}, differentiate the former with respect to $r$.

    For \eqref{eq:harmonic:diff1c} and \eqref{eq:harmonic:diff2c} we use \eqref{thm1_cor_id4} and the following evaluation:
    \begin{align*}
        \int_0^1(1-x)^rx^{s-1}(1+x)^ndx&=\sum_{k=0}^n\binom{n}{k}\int_0^1(1-x)^rx^{k+s-1}dx\\
        &=\sum_{k=0}^n\frac{\binom{n}{k}}{(k+s)\binom{k+s+r}{k+s}}.
    \end{align*}
\end{proof}


\begin{corollary}
    The following identities hold:
    $$\sum_{k=0}^n\sum_{j=0}^k(-1)^j\frac{\binom{n}{j}}{j+1}=\frac{n+2}{n+1}$$
    and
    $$\sum_{k=0}^n\sum_{j=0}^k(-1)^j\frac{\binom{n}{j}H_{j+1}}{j+1}=\frac{1}{n}+\frac{1}{(n+1)^2}.$$
\end{corollary}
\begin{proof}
    Set $r=s=1$ in \eqref{eq:harmonic_diff1b} and \eqref{eq:harmonic_diff2b}.
\end{proof}

\begin{corollary}
    The following identities hold::
    $$\sum_{k=0}^n\sum_{j=0}^k (-1)^k\frac{\binom{n}{j}}{j+1}=\frac{1+(-1)^n(2^{n+1}-1)}{2n+2}$$
    and
    $$\sum_{k=0}^n\sum_{j=0}^k\binom{n}{j}(-1)^k\frac{H_{j+1}}{j+1}=\frac{1}{2(n+1)^2}+\frac{(-2)^n}{n+1}\left(H_{n+1}-\sum_{k=1}^{n+1}\frac{1}{k2^k}\right).$$
\end{corollary}
\begin{proof}
    Set $r=s=1$ in \eqref{eq:harmonic:diff1c} and \eqref{eq:harmonic:diff2c}. Use the standard sum
    $$\sum_{k=0}^n\frac{\binom{n}{k}}{k+1}=\frac{2^{n+1}-1}{n+1}$$
    and the identity
    \begin{equation}
     \label{eq:SofoBoya}
        \sum_{k=0}^n \binom{n}{k} H_k = 2^n\Big ( H_n - \sum_{k=1}^{n} \frac{1}{k 2^k}\Big ) 
    \end{equation}
    (see \cite{BatirSofo2021,Boya1})
    in conjunction with
    $$\sum_{k=0}^n\frac{\binom{n}{k}H_{k+1}}{k+1}=\frac{1}{n+1}\sum_{k=0}^{n+1}\binom{n+1}{k}H_k.$$
\end{proof}

\begin{remark}
Identity \eqref{eq:evalxy=1OfCor36} can be simplified via \eqref{eq:SofoBoya} to the unusual identity
\begin{equation}
\sum_{k=1}^n \frac{1}{k} \sum_{j=0}^k \binom{n}{j} = \sum_{k=1}^n \frac{2^{n-k}+\binom{n}{k}}{k}.
\end{equation}
\end{remark}

More identities with harmonic numbers are possible. For instance, we can start directly with \eqref{main_id1} and perform computation similar to the one in the opening of this section. This can lead to the following.

\begin{proposition}\label{prop:general_harmonic_ids}
    For any complex $y$ and any admissible $r,s$ we have
    \begin{equation}\label{eq:Prop_1_y}
        \sum_{k=0}^n\sum_{j=0}^k\frac{\binom{n}{j}}{(k+s)\binom{k+r}{k+s}}y^j=\sum_{k=0}^n\frac{\binom{n}{k}}{(k+s)\binom{k+r-1}{k+s}}y^k-\frac{(1+y)^n}{(n+s+1)\binom{n+r}{n+s+1}}
    \end{equation}
    and
    \begin{equation*}
        \sum_{k=0}^n\sum_{j=0}^k\frac{\binom{n}{j}(H_{r-s}-H_{k+r})}{(k+s)\binom{k+r}{k+s}}y^j=\sum_{k=0}^n\frac{\binom{n}{k}(H_{r-1-s}-H_{k+r-1})}{(k+s)\binom{k+r-1}{k+s}}y^k-\frac{(1+y)^n(H_{r-s-1}-H_{n+r})}{(n+s+1)\binom{n+r}{n+s+1}}.
    \end{equation*}
\end{proposition}
\begin{proof}
    Start with \eqref{main_id1} by expanding the right-hand side:
    $$\sum_{k=0}^nx^k\sum_{j=0}^k\binom{n}{j}y^j=\frac{1}{1-x}\sum_{k=0}^n\binom{n}{k}x^ky^k-\frac{x^{n+1}}{1-x}(1+y)^n.$$
    Multiply that by $(1-x)^rx^{s-1}$ and use the Beta integral formula.
\end{proof}

\begin{corollary}
    The following identities hold:
    \begin{equation*}
        \sum_{k=0}^n\sum_{j=0}^k \frac{\binom{n}{j}}{(k+1)(k+2)}y^j=\frac{(1+y)^{n+1}-1}{y(n+1)}-\frac{(1+y)^n}{n+2}
    \end{equation*}
    and
    \begin{equation}\label{eq:cor_randomzxczcadq}
        \sum_{k=0}^n\sum_{j=0}^k \frac{\binom{n}{j}(H_{k+2}-1)}{(k+1)(k+2)}y^j=\sum_{k=0}^n\frac{\binom{n}{k}H_{k+1}}{k+1}y^k-\frac{(1+y)^nH_{n+2}}{n+2},
    \end{equation}
    and in particular
    \begin{equation*}
        \sum_{k=0}^n\sum_{j=0}^k \frac{\binom{n}{j}(H_{k+2}-1)}{(k+1)(k+2)}=\frac{2^{n+1}}{n+1}\left(H_{n+1}-\sum_{k=1}^{n+1}\frac{1}{k2^k}\right)-\frac{2^nH_{n+2}}{n+2},
    \end{equation*}
\end{corollary}
\begin{proof}
    Set $r=2$ and $s=1$ in Proposition \ref{prop:general_harmonic_ids}. Particular case is $y=1$.
\end{proof}

Finally, we note that Proposition \ref{prop:general_harmonic_ids} can be used to generate a variety of identities with Fibonacci numbers as well. Below, we showcase some of the numerous possibilities.
\begin{corollary}
    The following identities hold:
    \begin{align*}
          \sum_{k=0}^n\sum_{j=0}^k\frac{\binom{n}{j}F_{j+t}}{(k+s)\binom{k+r}{k+s}}&=\sum_{k=0}^n\frac{\binom{n}{k}F_{k+t}}{(k+s)\binom{k+r-1}{k+s}}-\frac{F_{2n+t}}{(n+s+1)\binom{n+r}{n+s+1}},\\
          \sum_{k=0}^n\sum_{j=0}^k\frac{\binom{n}{j}F_{3j+t}}{(k+s)\binom{k+r}{k+s}}&=\sum_{k=0}^n\frac{\binom{n}{k}F_{3k+t}}{(k+s)\binom{k+r-1}{k+s}}-\frac{2^nF_{2n+t}}{(n+s+1)\binom{n+r}{n+s+1}},\\
          \sum_{k=0}^n\sum_{j=0}^k\frac{\binom{n}{j}F_{4uj+t}}{(k+s)\binom{k+r}{k+s}}&=\sum_{k=0}^n\frac{\binom{n}{k}F_{4uk+t}}{(k+s)\binom{k+r-1}{k+s}}-\frac{L_{2u}^nF_{2un+t}}{(n+s+1)\binom{n+r}{n+s+1}},\\
        \sum_{k=0}^n\sum_{j=0}^k \frac{\binom{n}{j}(H_{k+2}-1)F_{j+t}}{(k+1)(k+2)}&=\sum_{k=0}^n\frac{\binom{n}{k}H_{k+1}F_{k+t}}{k+1}-\frac{F_{2n+t}H_{n+2}}{n+2}.
    \end{align*}
\end{corollary}
\begin{proof}
Set, in turn, $y=\alpha$, $y=\alpha^3$, $y=\alpha^{4u}$ in \eqref{eq:Prop_1_y} and $y=\alpha$ in \eqref{eq:cor_randomzxczcadq} (and similar substitutions with $\beta$).
\end{proof}

\section{Double sums evaluating to harmonic numbers}

\begin{lemma}\label{lem.j7mj9yc}
If $r$ is a complex number that is not a negative integer, then
\begin{align}
\int_0^1 {x^r \ln \left( {1 - x} \right)dx} &= - \frac{{H_{r + 1} }}{{r + 1}},\label{jd29oms}\\
\int_0^1 {\left( {1 - x} \right)^r \ln \left( {1 - x} \right)dx} &= - \frac{1}{{\left( {1 + r} \right)^2 }}\label{jgfp3eo},\\ \nonumber
\int_0^1 {x^r \ln ^2 \left( {1 - x} \right)dx} &= \frac{{H_{r + 1}^2 + H_{r + 1}^{(2)} }}{{r + 1}},\\ \nonumber
\int_0^1 {\left( {1 - x} \right)^r \ln ^2 \left( {1 - x} \right)dx} &= \frac{2}{{\left( {1 + r} \right)^3 }},\\ \nonumber
\int_0^1 {x^r \ln ^3 \left( {1 - x} \right)dx} &= - \frac{{H_{r + 1}^3 + 3H_{r + 1} H_{r + 1}^{(2)} + 2H_{r + 1}^3 }}{{r + 1}},
\end{align}
and
\begin{equation*}
\int_0^1 {\left( {1 - x} \right)^r \ln ^3 \left( {1 - x} \right)dx} = - \frac{6}{{\left( {1 + r} \right)^4 }}.
\end{equation*}
\end{lemma}
\begin{proof}
These identities, some of which are well-known, are readily obtained by the repeated differentiation of the Beta function,
\begin{equation}\label{er5offg}
\int_0^1 {x^u \left( {1 - x} \right)^v dx} = \frac{1}{{\binom{{u + v + 1}}{{v + 1}}\left( {v + 1} \right)}},\quad\Re u>-1,\,\Re v>-1,
\end{equation}
with respect to the parameters $u$ and $v$.
\end{proof}

\begin{proposition}
If $n$ is a non-negative integer, then
\begin{align}\nonumber
\sum_{k = 0}^n \sum_{j = 0}^k {\frac{{( - 1)^j \binom{{n}}{j}n}}{{n + j - k + 1}}} &= 1,\quad n\ne 0,\\
\sum_{k = 0}^n \sum_{j = 0}^k {\frac{{( - 1)^j \binom{n}{j} n}}{{\left( {n + j - k + 1} \right)^2 }}} &= H_n \label{uvy4my6},\\
\nonumber \sum_{k = 0}^n \sum_{j = 0}^k \frac{(- 1)^j \binom{{n}}{j}n}{\left( {n + j - k + 1} \right)^3} &= \frac{1}{2}\left(H_n^2 + H_n^{(2)} \right),
\end{align}
and
\begin{equation}
\sum_{k = 0}^n \sum_{j = 0}^k {\frac{{( - 1)^j \binom{{n}}{j}n }}{{\left( {n + j - k + 1} \right)^4 }}} 
= \frac{1}{6}\left( {H_n^3 + 3H_n H_n^{(2)} + 2H_n^{(3)} } \right)\label{f30y1gk}.
\end{equation}
\end{proposition}
\begin{proof}
These results follow from multiplying both sides of~\eqref{thm1_cor_id6} by $\ln^r(1-x)$, $r=0,1,2,3$ and termwise integrating 
from $0$ to $1$, using Lemma~\ref{lem.j7mj9yc}.
\end{proof}

\begin{remark}
Identity~\eqref{uvy4my6} can be taken as a double-sum definition of harmonic numbers.
\end{remark}
\begin{remark}
There is a dual identity to each of identities~\eqref{uvy4my6}--\eqref{f30y1gk}. We leave their calculation to the reader.
\end{remark}

\begin{lemma}
If $r$ is a complex number that is not a negative integer, then
\begin{equation}\label{rzep6h8}
\int_0^1 {x^r \left( {\Li_2 (x) + \ln x\ln (1 - x)} \right)\,dx} = \frac{H_{r + 1}^{(2)}}{{r + 1}}
\end{equation}
and
\begin{equation}\label{mmrsvxk}
\int_0^1 {\left( {1 - x} \right)^r \left( {\Li_2 (x) + \ln x\ln (1 - x)} \right)\,dx} = \frac{H_r}{{\left( {r + 1} \right)^2 }}.
\end{equation}
\end{lemma}
\begin{proof}
The following integrals are known (see \cite{Adegoke4}):
\begin{align*}
\int_0^1 {x^r \Li_2 (x)\,dx} = \frac{{\pi ^2 }}{{6(r + 1)}} - \frac{{H_{r + 1} }}{{\left( {r + 1} \right)^2 }},\\
\int_0^1 {\left( {1 - x} \right)^r \Li_2 (x)\,dx} = \frac{{\pi ^2 }}{{6(r + 1)}} - \frac{{H_{r + 1}^{(2)} }}{{r + 1}},
\end{align*}
and
\begin{equation*}
\int_0^1 {x^r \ln x\ln \left( {1 - x} \right)\,dx} = - \frac{{\pi ^2 }}{{6(r + 1)}} + \frac{{H_{r + 1}^{(2)} }}{{r + 1}} + \frac{{H_{r + 1} }}{{\left( {r + 1} \right)^2 }};
\end{equation*}
and hence~\eqref{rzep6h8} and~\eqref{mmrsvxk}.
\end{proof}

\begin{proposition}
If $n$ is a non-negative integer, then
\begin{equation*}
\sum_{k = 0}^n {\sum_{j = 0}^k {\frac{{( - 1)^j \binom{{n}}{j}nH_{n + j - k + 1} }}{{\left( {n + j - k + 1} \right)^2 }}} } = H_n^{(2)}
\end{equation*}
and
\begin{equation*}
\sum_{k = 0}^n {\sum_{j = 0}^k {\frac{{( - 1)^j \binom{{n}}{j} n^2 H_{n + j - k + 1}^{(2)} }}{{{n + j - k + 1}  }}} } = H_n.
\end{equation*}
\end{proposition}
\begin{proof}
Multiply both sides of~\eqref{thm1_cor_id6} by $\left( {\Li_2 (y) + \ln y\ln (1 - y)} \right)$ and termwise integrate from $0$ to $1$ using~\eqref{rzep6h8} and~\eqref{mmrsvxk}.
\end{proof}

\begin{lemma}\label{lem.j5d7kdb}
If $r$ is a non-negative integer, then
\begin{equation}\label{hiigp9e}
\int_0^1 {\left( {1 + x} \right)^r \ln \left( {1 + x} \right)dx} = \frac{{1 - 2^{r + 1} }}{{\left( {r + 1} \right)^2 }} + \frac{{2^{r + 1} }}{{r + 1}}\ln 2
\end{equation}
and
\begin{equation}\label{mhlgo8i}
\int_0^1 {x^r \ln \left( {1 + x} \right)dx}
=\frac1{r+1} 
\begin{cases}
 H_{r + 1}  - H_{(r + 1)/2} ,&\text{if $r$ is odd;} \\ 
 H_{r + 1}  - 2\,O_{(r + 2)/2}  + 2\ln 2,&\text{if $r$ is even.} \\ 
 \end{cases} 
\end{equation}
\end{lemma}
\begin{proof}
Identity~\eqref{hiigp9e} is obvious. A simple change of variable in~\eqref{jd29oms} gives
\begin{equation*}
\int_0^1 {x^r \ln \left( {1 + x} \right)dx}  = \frac{{H_{r + 1}  - H_{(r + 1)/2} }}{{r + 1}},\quad r\in\mathbb C\setminus\mathbb Z^{-},
\end{equation*}
of which~\eqref{mhlgo8i} is a special case, since
\begin{equation*}
H_{k+1/2}=2\,O_{k+1}-2\ln 2.
\end{equation*}
\end{proof}

\begin{proposition}
If $n$ is a non-negative integer, then
\begin{equation}\label{f41n8tu}
\sum_{k = 0}^n \sum_{j = 0}^k {\frac{{( - 1)^j \binom{{n}}{j} n 2^{j - k} }}{{\left( {n + j - k + 1} \right)^2 }}} = 2^{- n} O_{\left\lceil {n/2} \right\rceil } 
\end{equation}
and
\begin{equation}\label{qlnnrob}
\sum_{k = 0}^n {\sum_{j = 0}^k {\frac{{( - 1)^j \binom{{n}}{j}n2^{j - k} }}{{n + j - k + 1}}} }
= \begin{cases}
 0,&\text{if $n$ is even;} \\ 
 2^{ - n},&\text{if $n$ is odd.}  \\ 
 \end{cases} 
\end{equation}
\end{proposition}

\begin{proof}
Write~\eqref{thm1_cor_id6} as
\begin{equation*}
\sum_{k = 0}^n {\sum_{j = 0}^k {( - 1)^j \binom{{n}}{j}\left( {1 + x} \right)^{n + j - k} } } = ( - 1)^{n - 1} x^{n - 1} 
\end{equation*}
and termwise integrate both sides from $0$ to $1$, using Lemma~\ref{lem.j5d7kdb} to obtain
\begin{align*}
\sum_{k = 0}^n {\sum_{j = 0}^k {( - 1)^j \binom{{n}}{j}\frac{{1 - 2^{n + j - k + 1} }}{{\left( {n + j - k + 1} \right)^2 }}} }  + \ln 2\sum_{k = 0}^n {\sum_{j = 0}^k {( - 1)^j \binom{{n}}{j}\frac{{2^{n + j - k + 1} }}{{n + j - k + 1}}} } \\
 = \frac{( - 1)^{n - 1} }{n}
 \begin{cases}
 H_n  - H_{n/2},&\text{if $n$ is even;}  \\ 
 H_n  - 2O_{(n + 1)/2}  + 2\ln 2,&\text{if $n$ is odd;} \\ 
 \end{cases} 
\end{align*}
which upon equating the rational and irrational parts and simplifying gives~\eqref{f41n8tu} and~\eqref{qlnnrob}.
\end{proof}

\begin{proposition}
If $n$ is a non-negative integer, then
\begin{equation}\label{xdewshi}
\sum_{k = 0}^n {\sum_{j = 0}^k {\frac{{( - 1)^j \binom{{n}}{j}n}}{{\left( {2\left( {n + j - k} \right) + 1} \right)^2 }}} }  = 2^{2n - 1} \frac{{O_n }}{{\binom{{2n}}{n}}}.
\end{equation}
\end{proposition}
\begin{proof}
Set~\eqref{thm1_cor_id6} up as
\begin{equation}\label{jiryopw}
\sum_{k = 0}^n {\sum_{j = 0}^k {( - 1)^j \binom{{n}}{j}\int_0^1 {y^{2\left( {n + j - k} \right)} \ln y} } }\,dy  = \int_0^1 {\left( {1 - y^2 } \right)^{n - 1} \ln y}\,dy .
\end{equation}
Upon differentiating~\eqref{er5offg} with respect to the parameter $u$, making a change of variable $x=y^2$ and using $H_{k - 1/2}  - H_{ - 1/2}  = 2O_k$, it is not difficult to establish
\begin{equation*}
\int_0^1 {dy\,\left( {1 - y^2 } \right)^{n - 1} \ln y}  =  - \frac{{2^{2n - 1} }}{n}\frac{{O_n }}{{\binom{{2n}}{n}}},
\end{equation*}
while the integral on the left side of~\eqref{jiryopw} is evaluated using~\eqref{jgfp3eo}.
\end{proof}

\begin{remark}
Identity~\eqref{f41n8tu} or~\eqref{xdewshi} can be taken as a double-sum definition of odd harmonic numbers.
\end{remark}

\begin{remark}
The reader is invited to obtain the dual identity to~\eqref{xdewshi}.
\end{remark}

\section{Double sum identities involving Stirling numbers and $r$-Stirling numbers of the second kind}

\begin{lemma}\label{stirling2}
If $k$, $r$ and $v$ are non-negative integers, then
\begin{align}
\left. \frac{d^r}{dx^r}\left( 1 - e^x \right)^k \right|_{x = 0} &= \sum_{p = 0}^k (- 1)^p \binom{k}{p} p^r = (- 1)^k k!\braces rk,\label{eq.r33yojo}\\
\left. \frac{d^r}{dx^r}\left( 1 - e^x \right)^k e^{vx} \right|_{x = 0} &= \sum_{p = 0}^k (- 1)^p \binom{k}{p} (v + p)^r 
= (- 1)^k k!\braces{r+v}{k+v}_v  \label{eq.vompp34}.
\end{align}
\end{lemma}
\begin{proof}
Since
\begin{equation*}
\left( {1 - e^x } \right)^k e^{vx} = \sum_{p = 0}^k (- 1)^p \binom{{k}}{p}e^{(v + p)x};
\end{equation*}
we have
\begin{equation*}
\frac{d^r}{dx^r}\left( \left( {1 - e^x } \right)^k e^{vx} \right) = \sum_{p = 0}^k (- 1)^p \binom{{k}}{p}\left(v + p \right)^r e^{(v + p)x};
\end{equation*}
and hence~\eqref{eq.vompp34}.
\end{proof}

\begin{proposition}\label{prop.qemk345}
If $m$, $n$, $r$, and $s$ are non-negative integers, then
\begin{equation}\label{c2gonh5}
\sum_{k = 0}^n {( - 1)^k \sum_{j = 0}^k {\binom{{n}}{j}\braces{{ m + s}}{{n + j - k + r + s}}_s \left( {n + j - k + r} \right)!} }  = ( - 1)^n r!\braces{{ m + n + s - 1}}{{r + n + s - 1}}_{n + s - 1} .
\end{equation}
\end{proposition}
In particular,
\begin{equation}\label{lq8cg0u}
\sum_{k = 0}^n {( - 1)^k \sum_{j = 0}^k {\binom{{n}}{j}\braces{{ m}}{{n + j - k}}\left( {n + j - k} \right)!} }  = ( - 1)^n \left( {n - 1} \right)^m .
\end{equation}
\begin{proof}
Write~\eqref{thm1_cor_id6} as
\begin{equation*}
\sum_{k = 0}^n \sum_{j = 0}^k {( - 1)^j \binom{{n}}{j}\left( {1 - y} \right)^{n + j - k + r} y^s } = \left( {1 - y} \right)^r y^{n + s - 1} ,
\end{equation*}
write $\exp x$ for $y$, differentiate $m$ times with respect to $x$ and apply Lemma~\ref{stirling2}. Identity~\eqref{lq8cg0u} is obtain by setting $r=0=s$ in~\eqref{c2gonh5} and using the fact that
\begin{equation*}
\braces{m+n}{n}_n = n^m.
\end{equation*}
\end{proof}

\begin{proposition}
If $m$, $r$, $s$, $u$, and $v$ are non-negative integers, then
\begin{align*}
&\sum_{k = 1}^n {( - 1)^k \left( {k + u} \right)!\braces{{ m + s}}{{k + r + s}}_s \sum_{j = 1}^k {( - 1)^{j + u} \frac{{\binom{{n}}{j}}}{j}\braces{{ m + v}}{{j + u + v}}_v \left( {j + u} \right)!} }\nonumber\\ 
&\qquad = \sum_{k = 1}^n {( - 1)^k \frac{{\binom{{n}}{k}}}{k}\left( {k + u} \right)!\braces{{ m + v}}{{k + u + v}}_v \left( {( - 1)^{k + r} \left( {k + r} \right)!\braces{{ m + s - 1}}{{k + r + s - 1}}_{s -1} } \right.}\nonumber \\
&\qquad\qquad\qquad \left. { - ( - 1)^{n + r + 1} (n + r + 1)!\braces{{ m + s - 1}}{{n + r + s}}_{s-1} } \right)
\end{align*}
In particular,
\begin{align*}
&\sum_{k = 1}^n {( - 1)^k k!\braces{{ m + 1}}{k + 1}\sum_{j = 1}^k {\frac{( - 1)^j}j\binom nj\braces{{ m}}{j}j!} } \nonumber \\
&\qquad  = \sum_{k = 1}^n {\frac{( - 1)^k}k\binom nk\braces{{ m}}{k}k!\left( {( - 1)^k k!\braces{{ m}}{k} - ( - 1)^{n + 1} (n + 1)!\braces{{ m}}{{n + 1}}} \right)} ,
\end{align*}
with the special value
\begin{equation*}
\sum_{k = 1}^n ( - 1)^k k!\braces{{ n + 1}}{{k + 1}}\sum_{j = 1}^k {\frac{( - 1)^j}j \binom nj\braces{{ n}}{j}j!} 
= \sum_{k = 1}^n {\frac1k\binom nk\braces{{ n}}{k}^2 k!^2 } .
\end{equation*}
\end{proposition}
\begin{proof}
Use~\eqref{thm1_cor_id7} and proceed as in the proof of Proposition~\ref{prop.qemk345}.
\end{proof}

\begin{proposition}
If $m$, $r$, $s$, $u$, and $v$ are non-negative integers, then
\begin{align*}
&\sum_{k = 1}^n {\frac{{( - 1)^k }}{k}\braces{{ m + s}}{{k + r + s}}_s (k + r)!\sum_{j = 0}^k {\binom{{n}}{j}\braces{{ m + v}}{{j + u + v}}_v } } (j + u)!\nonumber\\
&\qquad = u!\braces{{ m + n + v}}{{u + n + v}}_{n + v} \sum_{k = 1}^n {\frac{{( - 1)^k }}{k}\braces{{ m + s}}{{k + r + s}}_s (k + r)!} \nonumber\\
&\qquad\qquad - \sum_{k = 1}^n {\binom{{n}}{k}\braces{{ m + v}}{{k + u + v}}_v(k + u)!\sum_{j = 1}^{k - 1} {\frac{{( - 1)^j }}{j}\braces{{ m + s}}{{j + r + s}}_s } (j+r)!} .
\end{align*}
\end{proposition}
\begin{proof}
Use~\eqref{thm1_cor_id9} and proceed as in the proof of Proposition~\ref{prop.qemk345}.
\end{proof}
In particular,
\begin{align*}
&\sum_{k = 1}^n {\frac{{( - 1)^k }}{k}\braces{{ m}}{{k + r}}(k + r)!\sum_{j = 0}^k {\binom{{n}}{j}\braces{{ m}}{{j + u}}} } (j + u)!\nonumber\\
&\qquad= u!\braces{{ m + n}}{u + n}_n \sum_{k = 1}^n {\frac{{( - 1)^k }}{k}\braces{{ m}}{{k + r}}(k + r)!} \nonumber\\
&\qquad\qquad - \sum_{k = 1}^n {\binom{{n}}{k}\braces{{ m + u}}{{k + u}}(k + u)!\sum_{j = 1}^{k - 1} {\frac{{( - 1)^j }}{j}\braces{{ m}}{{j + r}}(j+r)!} } 
\end{align*}
and
\begin{align*}
&\sum_{k = 1}^n {\frac{{( - 1)^k }}{k}\braces{{ m + 1}}{{k + r + 1}} (k + r)!\sum_{j = 0}^k {\binom{{n}}{j}\braces{{ m + 1}}{{j + u + 1}} } } (j + u)!\nonumber\\
&\qquad = u!\braces{{ m + n + 1}}{{u + n + 1}}_{n + 1} \sum_{k = 1}^n {\frac{{( - 1)^k }}{k}\braces{{ m + 1}}{{k + r + 1}} (k + r)!} \nonumber\\
&\qquad\qquad - \sum_{k = 1}^n {\binom{{n}}{k}\braces{{ m + 1}}{{k + u + 1}}(k + u)!\sum_{j = 1}^{k - 1} {\frac{{( - 1)^j }}{j}\braces{{ m + 1}}{{j + r + 1}} } (j+r)!} ,
\end{align*}
with the special values
\begin{align*}
\sum_{k = 1}^n {\frac{{( - 1)^k }}{k}\braces{{ m}}{{k}}k!\sum_{j = 0}^k {\binom{{n}}{j}\braces{{ m}}{{j }}} } j!&= n^m \sum_{k = 1}^n {\frac{{( - 1)^k }}{k}\braces{{ m}}{{k}}k!} \nonumber\\
&\qquad - \sum_{k = 1}^n \binom{{n}}{k}\braces{{ m }}{{k }}k!\sum_{j = 1}^{k - 1} {\frac{{( - 1)^j }}{j}\braces{{ m}}{{j}}j!} 
\end{align*}
and
\begin{align*}
\sum_{k = 1}^n {\frac{{( - 1)^k }}{k}\braces{{ m + 1}}{{k + 1}} k!\sum_{j = 0}^k {\binom{{n}}{j}\braces{{ m + 1}}{{j  + 1}} } } j!&=(n+1)^m \sum_{k = 1}^n {\frac{{( - 1)^k }}{k}\braces{{ m + 1}}{{k + 1}} k!} \nonumber\\
&\qquad - \sum_{k = 1}^n \binom{{n}}{k}\braces{{ m + 1}}{{k  + 1}}k!\sum_{j = 1}^{k - 1} {\frac{{( - 1)^j }}{j}\braces{{ m + 1}}{{j + 1}} } j!.
\end{align*}

\begin{proposition}
If $m$, $n$, $r$, and $s$ are non-negative integers, then
\begin{align*}
\sum_{k = 1}^n {\frac{1}{k}\sum_{j = 0}^k {\binom{{n}}{j}\binom{{j + s}}{s}\braces{{ m + r}}{{j + s + r}}_r j!} }  &= \braces{{ m + r + n}}{{s + r + n}}_{r + n} H_n\nonumber\\
&\qquad  - \sum_{k = 1}^n {\binom{{n}}{k}\binom{{k + s}}{s}\braces{{ m + r}}{{k + s + r}}_r k!H_{k - 1} }.
\end{align*}
In particular,
\begin{equation*}
\sum_{k = 1}^n \frac{1}{k} \sum_{j = 0}^k {\binom{n}{j}\braces{m}{j}j!} = n^m H_n - \sum_{k = 1}^m \binom{n}{k}\braces{m}{k} k! H_{k - 1},
\end{equation*}
with the special result
\begin{equation*}
\sum_{k = 1}^n \frac{1}{k} \sum_{j = 0}^k {\binom{n}{j} \braces{n}{j}j!} = n^n H_n - \sum_{k = 1}^n \binom{n}{k}\braces{n}{k} k! H_{k - 1}.
\end{equation*}
\end{proposition}
\begin{proof}
Use~\eqref{aqejkrz} and proceed as in the proof of Proposition~\ref{prop.qemk345}.
\end{proof}

\begin{remark}
Obviously, many more identities involving Stirling numbers and $r$-Stirling numbers of the second kind can be derived from the various identities. 
\end{remark}

\section{Concluding comments}

In this paper, we have presented a completely elementary approach to binomial double sums with an incomplete inner sum. 
To highlight the broad applicability of the main results, we have stated a range of new double sum identities involving prominent sequences of numbers. It should not come as a surprise that the main identity and its derivatives can also be used to formulate identities involving polynomials, and even the Horadam sequence which contains many polynomial sequences as special cases (see the survey paper \cite{larcombe}; see also \cite[Chapter 1]{ribenboim} for the properties of Lucas sequence). Recall that the Horadam sequence, $(w_j)_{j\in\mathbb Z} = \left(w_j(w_0,w_1;p,q)\right)$ is defined for all integers and arbitrary complex numbers $w_0$, $w_1$, $p\ne 0$ and $q\ne 0$, by the recurrence relation
\begin{equation*}
w_j = pw_{j - 1} - qw_{j - 2}, \quad  j \ge 2,
\end{equation*}
with $w_{- j}  =\left(pw_{-j + 1} - w_{-j + 2}\right)/q$. Associated with $(w_j)$ are the Lucas sequences of the first kind, 
$u_j(p,q)=w_j(0,1;p,q)$, and of the second kind, $v_j(p,q)=w_j(2,p;p,q)$. Our ideas can be applied to $w_j$ as well
and this is showcased in the final two propositions.

\begin{proposition}
If $n$ is a non-negative integer and $m$, $r$ and $t$ are integers, then
\begin{align*}
\sum_{k = 0}^n \sum_{j = 0}^k (- 1)^j \binom{{n}}{j} \binom{{n + j - k}}{r} v_m^{k - j} w_{t + m(n + j - k - r)} 
= (- 1)^r q^{m(n - r - 1)} v_m w_{t - m(n - r - 1)}.
\end{align*}
\end{proposition}

\begin{proposition}
If $n$ is a non-negative integer and $m$, $r$ and $t$ are integers, then
\begin{align*}
&\sum_{k = 0}^n \sum_{j = 0}^k (- 1)^j \binom{{n - r}}{{j - r}} v_m^{n - j} w_{t + m(j - r)} \nonumber \\ 
&\qquad = (- 1)^r \left( (n - r) v_m q^{m(n - r - 1)} w_{t - m(n - r - 1)} + q^{m(n - r)} w_{t - m(n - r)} \right).
\end{align*}
\end{proposition}

As another new field of application we will study the consequences of our main results to binomial transform pairs.
The results of this study will be presented in an upcoming article.

\end{document}